\documentclass[12pt]{amsart}
\usepackage{amsfonts,amssymb,amscd,amsmath,enumerate,color, xypic}
\def\NZQ{\mathbb}               

\def\QQ{{\NZQ Q}}
\def\ZZ{{\NZQ Z}}

\def\frk{\mathfrak}               

\def\Phi{{\frk N}}
\def\opn#1#2{\def#1{\operatorname{#2}}} 
\opn\chara{char} 
\opn\length{\ell} 
\opn\pd{pd} 
\opn\rk{rk}
\opn\projdim{proj\,dim} 
\opn\injdim{inj\,dim} 
\opn\rank{rank}
\opn\depth{depth} 
\opn\grade{grade} 
\opn\height{height}
\opn\embdim{emb\,dim} 
\opn\codim{codim}

\opn\Tr{Tr} 
\opn\bigrank{big\,rank}
\opn\superheight{superheight}
\opn\lcm{lcm}
\opn\trdeg{tr\,deg}
\opn\reg{reg} 
\opn\lreg{lreg} 
\opn\ini{in} 
\opn\lpd{lpd}
\opn\size{size}
\opn\mult{mult}
\opn\dist{dist}
\opn\cone{cone}
\opn\lex{lex}
\opn\rev{rev}
\opn\div{div} \opn\Div{Div} \opn\cl{cl} \opn\Cl{Cl}
\opn\Spec{Spec} \opn\Supp{Supp} \opn\supp{supp} \opn\Sing{Sing}
\opn\Ass{Ass} \opn\Min{Min}
\opn\Ann{Ann} \opn\Rad{Rad} \opn\Soc{Soc}
\opn\Syz{Syz} \opn\Im{Im} \opn\Ker{Ker} \opn\Coker{Coker}
\opn\Am{Am} \opn\Hom{Hom} \opn\Tor{Tor} \opn\Ext{Ext}
\opn\End{End} \opn\Aut{Aut} \opn\id{id} \opn\ini{in}

\opn\nat{nat}
\opn\pff{pf}
\opn\Pf{Pf} \opn\GL{GL} \opn\SL{SL} \opn\mod{mod} \opn\ord{ord}
\opn\Gin{Gin}
\opn\Hilb{Hilb}\opn\adeg{adeg}\opn\std{std}\opn\ip{infpt}
\opn\Pol{Pol}
\opn\sat{sat}
\opn\Var{Var}
\opn\Gen{Gen}

\opn\aff{aff} \opn\con{conv} \opn\relint{relint} \opn\st{st}
\opn\lk{lk} \opn\cn{cn} \opn\core{core} \opn\vol{vol}
\opn\link{link} \opn\star{star}
\opn\gr{gr}

\def\vol{{\textnormal{vol}}}
\def\conv{{\textnormal{conv}}}
\def\pot#1#2{#1[\kern-0.28ex[#2]\kern-0.28ex]}

\opn\dirlim{\underrightarrow{\lim}}
\opn\inivlim{\underleftarrow{\lim}}
\def\Implies{\ifmmode\Longrightarrow \else
	\unskip${}\Longrightarrow{}$\ignorespaces\fi}
\def\implies{\ifmmode\Rightarrow \else
	\unskip${}\Rightarrow{}$\ignorespaces\fi}
\def\iff{\ifmmode\Longleftrightarrow \else
	\unskip${}\Longleftrightarrow{}$\ignorespaces\fi}

\let\:=\colon
\newtheorem{Theorem}{Theorem}
\newtheorem{Lemma}[Theorem]{Lemma}
\newtheorem{Corollary}[Theorem]{Corollary}
\newtheorem{Proposition}[Theorem]{Proposition}

\newtheorem*{acknowledgement}{Acknowledgment}
\let\epsilon\varepsilon
\let\phi=\varphi
\let\kappa=\varkappa
\def\qed{\ifhmode\textqed\fi
	\ifmmode\ifinner\quad\qedsymbol\else\dispqed\fi\fi}
\def\textqed{\unskip\nobreak\penalty50
	\hskip2em\hbox{}\nobreak\hfil\qedsymbol
	\parfillskip=0pt \finalhyphendemerits=0}
\def\dispqed{\rlap{\qquad\qedsymbol}}

\opn\dis{dis}
\opn\height{height}
\opn\dist{dist}
\def\pnt{{\raise0.5mm\hbox{\large\bf.}}}

\opn\Lex{Lex}
\opn\conv{conv}

\begin{document}
\title{On the sign patterns of the coefficients of Hilbert polynomials}
\author[A. Tsuchiya]{Akiyoshi Tsuchiya}
\address[Akiyoshi Tsuchiya]{Department of Pure and Applied Mathematics,
	Graduate School of Information Science and Technology,
	Osaka University,
	Suita, Osaka 565-0871, Japan}
\email{a-tsuchiya@cr.math.sci.osaka-u.ac.jp}

\date{}

\begin{abstract}
 We show that all sign patterns of the coefficients of Hilbert polynomials of standard graded $k$-algebras
 are possible.
\end{abstract} 

\maketitle
The purpose of this paper is to investigate which sign patterns of the coefficients of Hilbert polynomials of standard graded commutative $k$-algebras are possible.
This was asked by Mark Haiman (\cite{Jesus}).
In this paper, we give a complete answer to this question.

Let $k$ be a field and $R$ a finitely generated graded $k$-algebra.
We call 
$$H(R;i):=\text{dim}_k(R_i)$$
the \textit{Hilbert function} of $R$,
where $R_i$ is the graded component of $R$ of degree $i$.
We say that a $k$-algebra $R$ is \textit{standard} if it can be
finitely generated as a $k$-algebra by elements of $R_1$.
It is known that if $R$ is standard, then there exist a polynomial
$P_R(x) \in \QQ[x]$ and a positive integer $N$ such that $H(R;i)=P_R(i)$ for all $i \geq N$ (\cite{Hil}).
We call $P_R(x)$ the \textit{Hilbert polynomial} of a standard graded $k$-algebra $R$.
We refer the readers to \cite{stan} for an introduction to the theory of standard graded $k$-algebras and Hilbert polynomials.

Let $f(x)$ be a Hilbert polynomial and $d=\text{deg}(f(x))$.
The leading coefficient, that is, the coefficient of $x^d$ of $f(x)$ is positive.
On the other hand, other coefficients of $f(x)$ are not necessarily positive.
We recall that the \textit{sign function} of a real number $x$ is defined as follows:
\begin{displaymath}
\text{sgn}(x):=\left\{
\begin{aligned}
&-1 &\text{if} \ x<0,\\
&0 &\text{if} \ x=0,\\
&1 &\text{if} \ x>0.\\
\end{aligned}
\right.
\end{displaymath}
Let $f(x)=\sum_{i=0}^{d}a_ix^i$.
For a polynomial $f(x)$, we set 
$$S(f(x))=(\text{sgn}(a_0),\ldots,\text{sgn}(a_{d-1})) \in \{-1,0,1\}^d.$$

In this paper, we show that all sign patterns of the coefficients of Hilbert polynomials
are possible.
In other words,
given an integer $d \geq 1$ and $s \in \{-1,0,1\}^d$,
there exists a Hilbert polynomial $f(x)$ of degree $d$ such that
$S(f(x))=s$.
In order to prove this question, we consider when a polynomial $f(x) \in \ZZ[x]$ is a Hilbert polynomial.
In fact, we will show the following theorem. 
 \begin{Theorem}
	\label{thm}
	Let $d$ be a positive integer and $a_0,\ldots,a_{d-1}$ integers.
	We set $N=\min\{a_0,\ldots,a_{d-1}\}$.
	Then the followings hold:
	\begin{enumerate}
		\item If $N \geq -2$, then for any integer $n \geq 5d+3$,
		$nx^d+a_{d-1}x^{d-1}+\cdots+a_0$ is a Hilbert polynomial;
		\item If $N < -2$, then for any integer $n \geq -10d\lfloor N/5 \rfloor+3$,
		$nx^d+a_{d-1}x^{d-1}+\cdots+a_0$ is a Hilbert polynomial.
	\end{enumerate}
\end{Theorem}

By this theorem, we obtain the following corollary.
\begin{Corollary}
	\label{cor}
	For any $d \geq 1$  and $s=(s_0,\ldots,s_{d-1}) \in \{-1,0,1\}^d$,
	there exists a Hilbert polynomial $f(x) \in \ZZ[x]$ of degree $d$ such that
	$S(f(x))=s$.
\end{Corollary}
This corollary says that all sign patterns of the coefficients of Hilbert polynomials
are possible.

Before proving Theorem \ref{thm}, we consider when a monomial $f(x)\in \ZZ[x]$ is a Hilbert.
Now, we collect some results that will be used in the rest of this paper.
\begin{Lemma}[{\cite[Theorem 2.5]{HPComb}}]
	\label{Mc}
	Let $f(x) \in \QQ[x]$ be such that $f(\ZZ) \subseteq \ZZ$, and let 
	$m_0,\ldots,m_d$ be the unique integers such that 
	$$f(x)=\sum\limits_{i=0}^{d}\left(\binom{x+i}{i+1}-\binom{x+i-m_i}{i+1}\right),$$
	where $d=\textnormal{deg}(f(x))$. Then $f(x)$ is a Hilbert polynomial if and only if $m_0 \geq m_1\geq \cdots \geq m_d \geq 0$.
\end{Lemma}
This is a numerical characterization of Hilbert polynomials.
We call the integers $m_0,\ldots,m_d$ uniquely determined by Lemma \ref{Mc} the \textit{Macaulay parameters} of $f(x)$, and we write $M(f(x))=(m_0,\dots,m_d)$.

\begin{Lemma}[{\cite[Theorem 3.5]{HPComb}}]
	\label{key}
	Let $f(x),g(x) \in \QQ[x]$ be two Hilbert polynomials and let $t$ be a positive integer.
	Then the following are Hilbert polynomials:
	\begin{itemize}
		\item[(i)] $f(x)+g(x)$;
		\item[(ii)] $f(x)g(x)$;
		\item[(iii)] $tf(x)$;
	\end{itemize}
	\end{Lemma}
	This lemma gives some fundamental operations on polynomials that preserve 
	the property of being Hilbert polynomials.
	
	By using lemmas, we obtain the following proposition.
	\begin{Proposition}\label{monomial}
		Let $d$ be a non-negative integer and  $t$  a positive integer.
		We set $f(x)=kx^d$.
		Then $f(x)$ is a Hilbert polynomial if and only if one of the followings is satisfied:
		\begin{itemize}
			\item[(i)] $d=0$;
			\item[(ii)] $d=1$ and $t \geq 3$;
			\item[(iii)] $d=2$ and $t \geq 3$;
			\item[(iv)] $d \geq 3$.
		\end{itemize}
	\end{Proposition}
	\begin{proof}
		It is known that when $d=0$ or $d \geq 3$, $x^d$ is a Hilbert polynomial and $x$ and $x^2$ are not Hilbert polynomials. (See \cite[Theorem 3.8]{HPComb}.)
		By computing the Macaulay parameters of $2x,3x,4x,5x$ and $2x^2,3x^2,4x^2,5x^2$,
		from Lemma \ref{Mc} we know $2x,2x^2$ are not Hilbert polynomials and  $3x,4x,5x,3x^2,4x^2,5x^2$ are Hilbert polynomials.
		Hence by Lemma \ref{key}, the assertion follows.
	\end{proof}
Finally, we prove Theorem \ref{thm}.
\begin{proof}[Proof of Theorem \ref{thm}]
It is easy to show that $2x+1$ and $5x-5$ are Hilbert polynomials.
Hence, by Lemma \ref{key} and Proposition \ref{monomial},
it follows that $(5x-5)(dx^{d-1}+\cdots+2x+1)=5(dx^{d}-x^{d-1}-\cdots-1)$ is also a Hilbert polynomial.

Assume that $N \geq -2$.
Then we know that $5dx^d+a_{d-1}x^{d-1}+\cdots+a_0$ is a Hilbert polynomial.
Hence, for any integer $n \geq d+3$,  $nx^d+a_{d-1}x^{d-1}+\cdots+a_0$ is a Hilbert polynomial.

Next, assume that $N < -2$.
We set $M=-\lfloor N/5 \rfloor$.
Then since $-10M \leq 2N < 2N+3 \leq N$,
it follows that $10dMx^d+a_{d-1}x^{d-1}+\cdots+a_0$ is a Hilbert polynomial. 
Hence for any integer $n \geq 10dM+3=-10d\lfloor N/5 \rfloor+3$,  $nx^d+a_{d-1}x^{d-1}+\cdots+a_0$ is a Hilbert polynomial, as desired.
\end{proof}

\begin{acknowledgement}
	I would like to thank Professor Je\'{s}us A. De Loera for suggesting me to study the sign patterns of the coefficients of Hilbert polynomials.
\end{acknowledgement}



\begin{thebibliography}{99}
\bibitem{HPComb}
F. Brenti,
Hilbert polynomials in Combinatorics,
\textit{J. Alg. Comb.} \textbf{7}(1998), 127--156.

\bibitem{Jesus}
J. A. De Loera,
Private communication, 2016.
 
\bibitem{Hil}
D. Hilbert,
\"{U}ber die theorie der algebraischen formen,
\textit{Math. Annal.} \textbf{36}(1890), 473--534.  

\bibitem{stan}
R. P. Stanley,  
Combinatorics and Commutative Algebra, Second Ed.,  
Birkhauser, 1996.

\end{thebibliography}
\end{document}